\documentclass[12pt]{iopart}

\expandafter\let\csname equation*\endcsname\relax
\expandafter\let\csname endequation*\endcsname\relax
  
\usepackage{amsmath}
\usepackage{amsfonts}
\usepackage{amssymb}
\usepackage{amsthm}
\usepackage{stmaryrd}
\usepackage{paralist}
\usepackage{bm}
\usepackage{cite}
\usepackage[english]{babel}
\usepackage{dblaccnt}
\usepackage{accents}
\usepackage{graphicx}
\usepackage{amscd}
\usepackage[mathscr]{euscript}
\usepackage{graphicx, color,float}
\usepackage[toc,page]{appendix}
\usepackage[hang,small,bf]{caption}
\usepackage{subfig}
\usepackage[shortlabels]{enumitem}
\usepackage[norefs,nocites]{refcheck}
\usepackage{autonum}
\usepackage{tikz}
\usepackage{comment}

\newcommand{\bfxZ} {\bfz}
\newcommand{\Tcal}{\mathcal{T}}
\newcommand{\Jcal}{\mathcal{J}}
\newcommand{\bfz} {\boldsymbol{z}}
\newcommand{\bfA} {\boldsymbol{A}}
\newcommand{\bfB} {\boldsymbol{B}}
\newcommand{\bfx} {\boldsymbol{x}}
\newcommand{\sip} {\! \cdot\!}
\newcommand{\bfxh} {\hat{\boldsymbol{x}}{}}
\newcommand{\bfC} {\boldsymbol{C}}
\newcommand{\dip} {\! :\!}
\newcommand{\sym}{_{\text{sym}}}
\newcommand{\Div}{\mbox{div}\,} 
\newcommand{\Lpar}{\Big(\,}
\newcommand{\Rpar}{\,\Big)}
\newcommand{\rpar}{\hspace*{0.1em}\big)}
\newcommand{\lpar}{\big(\hspace*{0.1em}}
\newcommand{\bfna}{\boldsymbol{\nabla}}
\newcommand{\Rbb} {\mathbb{R}}
\newcommand{\sheq}{\hspace*{-0.1em}=\hspace*{-0.1em}}
\newcommand{\bfAT} {\widetilde{\boldsymbol{A}}{}}
\newcommand{\shsetm}{\hspace*{-0.1em}\setminus\hspace*{-0.1em}}
\newcommand{\Gs}{\G_{s}}
\newcommand{\Gm}{\G_{m}}
\newcommand{\Rm}{R_{m}}
\newcommand{\Rs}{R_{s}}
\newcommand{\G}{\Gamma}
\newcommand{\shcap}{\hspace*{-0.1em}\cap\hspace*{-0.1em}}
\newcommand{\bfy} {\boldsymbol{y}}
\newcommand{\dbfy}{\,\text{d}\bfy}
\newcommand{\shm}{\hspace*{-0.1em}-\hspace*{-0.1em}}
\newcommand{\shin}{\hspace*{-0.1em}\in\hspace*{-0.1em}}
\newcommand{\bfrh} {\hat{\boldsymbol{r}}{}}
\newcommand{\bfr} {\boldsymbol{r}}
\newcommand{\shdeq}{\hspace*{-0.1em}:=\hspace*{-0.1em}}
\newcommand{\iGs} {\int_{\G_{s}}}
\newcommand{\bfs} {\boldsymbol{s}}
\newcommand{\dbfs}{\,\text{d}\bfs}
\newcommand{\obs}{_{\text{obs}}}
\newcommand{\demi} {\inv{2}}
\newcommand{\iGm} {\int_{\G_{m}}}
\newcommand{\labs}{\big|\hspace*{0.1em}}
\newcommand{\rabs}{\hspace*{0.1em}\big|}
\newcommand{\bfm} {\boldsymbol{m}}
\newcommand{\inv}[1]{\dfrac{1}{#1}}
\newcommand{\dbfm}{\,\text{d}\bfm}
\newcommand{\Bcal}{\mathcal{B}}
\newcommand{\OO}{\Omega}
\newcommand{\iBcal}{\int_{\Bcal}}
\newcommand{\Lcb}{\Big\{\hspace*{0.1em}}
\newcommand{\Rcb}{\hspace*{0.1em}\Big\}}
\newcommand{\dVx}{\,\text{d}V_{x}}
\newcommand{\ssup}{^{\text{s}}}
\newcommand{\shneq}{\hspace*{-0.1em}\not=\hspace*{-0.1em}}
\newcommand{\bfM} {\boldsymbol{M}}
\newcommand{\bfW} {\boldsymbol{W}}
\newcommand{\bfg} {\boldsymbol{g}}
\newcommand{\bfh} {\boldsymbol{h}}
\newcommand{\comp}{_{\text{comp}}}
\newcommand{\Cbb} {\mathbb{C}}
\newcommand{\bfT} {\boldsymbol{T}}
\newcommand{\bfI} {\boldsymbol{I}}
\newcommand{\iB}{\int_{B}}
\newcommand{\lsqb}{\big[\hspace*{0.1em}}
\newcommand{\rsqb}{\hspace*{0.1em}\big]}
\newcommand{\bfH} {\boldsymbol{H}}
\newcommand{\shsubs}{\hspace*{-0.1em}\subset\hspace*{-0.1em}}
\newcommand{\bfG} {\boldsymbol{G}}
\newcommand{\MS}{\mbox{\boldmath $\mathcal{M}$}}
\newcommand{\nnum}{\notag \\}
\newcommand{\Mcal}{\mathcal{M}}
\newcommand{\tens}{\hspace*{-1pt}\otimes\hspace*{-1pt}}
\newcommand{\derpq}[3]{\dfrac{\partial^{2} #1}{\partial #2 \partial #3}}
\newcommand{\iG} {\int_{\G}}
\newcommand{\eps}{\varepsilon}
\newcommand{\lcb}{\big\{\hspace*{0.1em}}
\newcommand{\rcb}{\hspace*{0.1em}\big\}}
\newcommand{\bfn} {\boldsymbol{n}}
\newcommand{\dV}{\;\text{d}V}
\newcommand{\dS}{\,\text{d}S}
\newcommand{\tila}{\tilde{a}}
\newcommand{\shp}{\hspace*{-0.1em}+\hspace*{-0.1em}}
\newcommand{\shl}{\hspace*{-0.1em}<\hspace*{-0.1em}}
\newcommand{\bfR} {\boldsymbol{R}}
\newcommand{\bfD} {\boldsymbol{D}}
\newcommand{\Tsup}{^{\text{\scriptsize T}}}
\newcommand{\RS}{\ensuremath{\mbox{\boldmath $\mathcal{R}$}}}
\newcommand{\bfK} {\boldsymbol{K}}
\newcommand{\Rcal}{\mathcal{R}}
\newcommand{\bfsh} {\hat{\boldsymbol{s}}{}}
\newcommand{\bfbh} {\hat{\boldsymbol{b}}{}}
\newcommand{\Lsqb}{\Big[\hspace*{0.1em}}
\newcommand{\Rsqb}{\hspace*{0.1em}\Big]}
\newcommand{\bfzh} {\hat{\boldsymbol{z}}{}}
\newcommand{\bfhb} {\bar{\boldsymbol{b}}{}}
\newcommand{\bfQ} {\boldsymbol{Q}}
\newcommand{\bfbe}{\boldsymbol{\beta}}
\newcommand{\bfsig}{\boldsymbol{\sigma}}
\newcommand{\bfq} {\boldsymbol{q}}
\newcommand{\bfS} {\boldsymbol{S}}
\newcommand{\hatS}{\hat{S}}
\newcommand{\bfyh} {\hat{\boldsymbol{y}}{}}
\newcommand{\suite}[1][0ex]{\notag \\[#1] & \mbox{}\hspace{15pt}}
\newcommand{\shcup}{\hspace*{-0.1em}\cup\hspace*{-0.1em}}
\newcommand{\jump}[1]{[\![ {#1} ]\!]}

\newtheorem{lemma}{Lemma}
\newtheorem{remark}{Remark}
\newtheorem{theorem}{Theorem}
\newcommand{\Phik}{\Phi_{\kappa}}
\newcommand{\Phiz}{\Phi_0}

\begin{document}

\title[Analysis of topological derivative as a tool for qualitative identification]{Analysis of topological derivative as a tool for qualitative identification}

\author{Marc Bonnet$^1$ and Fioralba Cakoni$^2$}
\address{$^1$ POEMS (CNRS, INRIA, ENSTA), ENSTA, 828 boulevard des Mar\'echaux, 91120 Palaiseau, France.}
\address{$^2$ Department of Mathematics, Rutgers University, 110 Frelinghuysen Road,  Piscataway, NJ 08854-8019, USA.}

\ead{mbonnet@ensta.fr, fc292@math.rutgers.edu}
\vspace{10pt}
\begin{indented}
\item[] 
\end{indented}
\begin{abstract}
The concept of topological derivative has proved effective as a qualitative inversion tool for a wave-based identification of finite-sized objects. Although for the most part, this approach remains  based on a heuristic interpretation of the topological derivative, a first attempt toward  its mathematical justification was done in Bellis et al. (\emph{Inverse Problems} \textbf{29}:075012, 2013) for the case of isotropic media with far field data and inhomogeneous refraction index. 
Our paper extends the analysis there to the case of anisotropic scatterers and background with near field data. Topological derivative-based imaging functional is analyzed using a suitable factorization of the near fields, which became achievable thanks to a new volume integral formulation recently obtained in Bonnet (\emph{J. Integral Equ. Appl.} \textbf{29}:271--295, 2017). Our results include justification of sign heuristics for the topological derivative in the isotropic case with jump in the main operator and for some cases of anisotropic media, as well as verifying its decaying property in the isotropic case with near field spherical measurements configuration situated far enough from the probing region.  
\end{abstract}

\section{Introduction}

Inverse scattering has undergone intense investigation over the last quarter century, in particular due to the growth and flourishing of qualitative methods which provide robust and computationally effective alternatives to more traditional approaches based on successive linearizations or PDE-constrained optimization, see~\cite{kirsch:grinberg,cakoni:colton:14,cakoni:16} for expository material and references. Qualitative identification methods usually consist in \emph{sampling} a spatial region of interest with points $\bfxZ$ at which an imaging function $\phi$ is evaluated; this is in particular the case for (generalized) linear sampling methods and factorization methods. The latter are moreover backed by firm and comprehensive mathematical justifications.

An alternative basis for qualitative identification is provided by the concept of \emph{topological derivative} (TD). The TD of an objective functional $\Jcal$ quantifies the leading perturbation to $\Jcal$ induced by the nucleation of a trial object of vanishingly small radius $\delta$ at a given location $\bfxZ$ in the background (i.e. defect-free) medium. On taking $\Jcal$ as a misfit functional of the kind typically used for inversion by PDE-constrained optimization, the value of the TD of $\Jcal$ at $\bfxZ$, herein denoted $\Tcal(\bfxZ)$, provides a basis for a sampling approach (by choosing $\phi(\bfxZ):=\Tcal(\bfxZ)$). The underlying heuristic idea is that $\Tcal(\bfxZ)$ is intuitively expected to take pronounced negative values at the correct location of a sought defect, consistently with the notion of minimizing $\Jcal$. This heuristic thus involves both the magnitude (expected to be largest) and the sign (expected to be negative) of $\Tcal(\bfxZ)$ for $\bfxZ$ near the defect support.

The idea of TD was initially introduced and formalized as a computational aid for topology optimization problems~\cite{Esch,soko:zoch:99}, and has thereafter also proved effective for revealing hidden objects in a variety of inverse scattering situations, see e.g.~\cite{B-2003-5,boj:07,B-2012-6,dom:gib:esq:05,agjk:11,B-2005-17,B-2016-12,laurain:13,lelouer:17}. In particular, despite the asymptotic character of the mathematical concept of TD, numerous available computational results show its ability to qualitatively identify spatially-extended objects. The objective functional $\Jcal$ underpinning $\Tcal(\bfxZ)$ in practice often expresses the misfit between data and its model prediction in a least-squares sense, which has the advantage of making TD-based imaging workable for any available data. Moreover, the practical evaluation of $\bfxZ\mapsto\Tcal(\bfxZ)$ only requires the incident field and an adjoint field~\cite{cea:mas}, so is both straightforward and moderately expensive from a computational standpoint.

The definition and formulation of $\Tcal(\bfxZ)$ for given physical setting and objective functional is a mathematically rigorous operation. By contrast, its subsequent application towards imaging defects by using the previously-described heuristics is still supported mainly by computational evidence and lacks a comprehensive mathematical foundation. Theoretical investigations about TD-based imaging have begun only recently. The imaging of a single small scatterer in an acoustic medium is mathematically studied in \cite{agjk:11}, where proofs of stability with respect to medium or measurement noises are also given; this framework has since then been extended to elastodynamics~\cite{ammari:13b} and electromagnetism~\cite{wahab:15}. The high-frequency limiting behavior of a TD imaging functional is analyzed in~\cite{guz:pour:15}. The qualitative identification of spatially extended objects, which is the main focus of this work, was first considered in~\cite{B-2013-01} for a rather idealized setting involving $L^2$ misfit cost functionals incorporating far-field data and scatterers characterized by a inhomogeneous refraction index. It was shown in that context that the magnitude component of the heuristic interpretation is valid without limitations, whereas the guaranteed correctness of the sign component is subject to an inequality (involving the operating frequency and the obstacle size and contrast) that essentially requires the scatterer to be ``moderate enough''.

In this work, we continue the line of investigation initiated in~\cite{B-2013-01} by considering situations where (i) the medium properties are characterized by a tensor-valued coefficient appearing in the principal, second-order term of the governing differential operator (rather than a refraction index affecting the zeroth-order term) and (ii) data is collected at a finite distance (rather than in the far field). The (uniform) host medium and the scatterer may both be anisotropic. Our main aim is to establish conditions under which the usual heuristic for TD imaging is valid. Towards this aim, we formulate the forward scattering problem as a volume integral equation, and take advantage of a recently-proposed reformulation of such volume integral equation ~\cite{B-2016-06} which allows to express $\Tcal(\bfxZ)$ separately in terms of  the material contrast and a contrast-independent normalized integral operator; this in particular facilitates the handling of material anisotropy. Some of our main findings are similar in nature to those of~\cite{B-2013-01}; in particular the sign component of the TD heuristic is again found to be valid within  a ``moderate enough scatterer'' condition, here expressed in terms of the norm of the normalized integral operator. We emphasize that this condition is less stringent than a requirement that the Born approximation be valid. Our other main contribution consists of an asymptotic study of the decay of $|\Tcal(\bfxZ)|$ when the sampling region spanned by $\bfxZ$ is large relative to the obstacle diameter while the measurements are taken far from the sampling region. The expected decay of $\bfxZ\mapsto|\Tcal(\bfxZ)|$ is as a result observed for far-field data (leading-order asymptotics), as expected from e.g.~\cite{B-2013-01}, but also on the next-order asymptotic contribution.
 
The article is organized as follows. In the next section we formulate the direct and inverse scattering problem for anisotropic media for near field data, and introduce the topological derivative as the first order coefficient in the asymptotic expansion of the cost function in terms of the size of the trial inhomogeneities. The excitation and measurement surfaces may not be the same and partial aperture data is allowed under some assumptions.  Only the fields inside the bounded region circumscribed by the excitation or  measurement surface (whatever bounds the larger region) matter in our analysis, hence the discussion presented here includes the case when the scattering problem is formulated in the whole space or in a bounded region, with obvious changes in the  fundamental solution. Section 3 is dedicated to the derivation of explicit expressions for the topological derivative, where a  new volume integral equation for anisotropic media recently obtained in \cite{B-2016-06}  plays an essential role in obtaining a symmetric factorization of TD.  We consider two cases: isotropic scatterers in Section 4 and anisotropic scatterer in Section 5. The study of the former is more complete, namely we provide the justification of the sign heuristic of TD restricted to scatterers of moderate strength in terms of scatterer size, its material contrast and the operating frequency, as well as show the decaying property of TD for sampling points far from the unknown inhomogeneity for spherical near field measurements configuration far enough form the scatterer. The case of anisotropic scatterers is more complicated and partial results on the justification of TD sign heuristic are obtained  in specialized cases such as for anisotropic scatterers embedded in isotropic background and scatterers of one-sign contrasts.  
\section{Formulation of the scattering problem and topological derivative}
We start by setting up some notation conventions which will be used throughout the paper. In expressions such as $\bfA\sip\bfx$ or $\bfB\dip\bfC$, symbols '${}\sip{}$' and '${}\dip{}$' denote single and double inner products, e.g. $(\bfA\sip\bfx)_i=A_{ij} x_j$ and $\bfB\dip\bfC \sheq B_{ij}C_{ij}$, with Einstein's convention of summation over repeated indices implicitly used throughout and component indices always referring to an orthonormal frame. The (Euclidean) norm of a vector or tensor $\bfx$ is denoted by $|\bfx|$, whereas $\|\cdot\|$ indicate norms in function spaces or operator norms. Hat symbols over vectors denote corresponding unit vectors, e.g. $\bfxh:=\bfx/|\bfx|$.

\subsection{Direct scattering problem}

We consider an unbounded, homogeneous reference propagation medium whose constitutive properties can be described by the real-valued symmetric tensor $\bfA\in\Rbb\sym^{3\times3}$, so that (in the absence of any sources in the medium) a propagating wave described by the complex-valued function $u$ satisfies 
\begin{equation}
   -\Div\lpar\bfA\sip\bfna u\rpar - \kappa^2 u = 0 \label{eq:propag}
\end{equation}
(see~\cite{dass:kara:05} for details on scattering in anisotropic media). The medium hosts an unknown inhomogeneity with compact support $B\subset\Rbb^3$ whose material properties are characterized by $\bfAT\in\Rbb\sym^{3\times3}$. Both $\bfA$ and $\bfAT$ are positive definite. The perturbed medium can then be characterized by $\bfA_B\in L^{\infty}(\Rbb^3;\Rbb\sym^{3\times3})$ such that
\[
  \bfA_B := \bfAT \quad \text{in $B$,} \qquad \bfA_B := \bfA \quad \text{in $\Rbb^3\shsetm\overline{B}$}
\]
Let $\Gs$ and $\Gm$ denote two closed surfaces, which respectively support probing excitations and measurements. We denote by $\Rm$ and $\Rs$ the bounded domains enclosed by $\Gm$ and $\Gs$. We will consider the following possibilities for the source / measurement configuration: (i) $\Rm=\Rs$, i.e. $\Gm=\Gs$; (ii) $\overline{\Rm}\Subset\Rs$, i.e. $\Gm$ is inside $\Gs$; (iii) $\overline{\Rs}\Subset\Rm$, i.e. $\Gs$ is inside $\Gm$. In all cases, $B\Subset R$, the \emph{region of interest} $R$ being defined by $R:=\Rs\shcap\Rm$, i.e. both $\Gs$ and $\Gm$ surround the unknown inhomogeneity.

\noindent
This work will make frequent use of single-layer potentials created by superpositions of sources on $\Gs$ or $\Gm$. Let the single-layer potential operator $S_{r\alpha}:H^{-1/2}(\G_{\alpha})\to H^1(R)$ ($\alpha\sheq{m,s}$) be defined by
\begin{equation}
  S_{r\alpha}\varphi(\bfx)
 = \int_{\G_{\alpha}} \Phik(\bfx\shm\bfy) \varphi(\bfy) \dbfy \qquad \bfx\shin R,\;\alpha\sheq{m,s}, \label{S:def}
\end{equation}
where $\Phik(\bfx\shm\bfy)$ is the fundamental solution for the background medium, satisfying 
\begin{equation}\label{fund:gov:eq}
  -\Div\lpar\bfA\sip\bfna \Phik(\bfx\shm\bfy)\rpar - \kappa^2 \Phik(\bfx\shm\bfy) = \delta(\bfx\shm\bfy) \qquad \bfx\in{\mathbb R}^3\setminus\{\bfy\}
\end{equation}  
  together with the outgoing radiation condition at infinity. For a generic wave $u$, the radiation condition involved in problems~\eqref{eq:propag} and~\eqref{S:def} is (see \cite{dass:kara:05})
  \begin{equation}
  \lpar \bfrh\sip\bfA^{-1}\sip\bfrh \rpar^{1/2} \bfrh\sip\bfA^{-1}\sip\bfna u(\bfr) - \mathrm{i}\kappa u(\bfr)
 = O(|\bfr|^{-2}) \qquad |\bfr|\to\infty,\label{rad:cond}
\end{equation}
and reduces to the usual Sommerfeld condition if the medium is isotropic. An explicit expression of  $\Phik$ is given in \cite{dass:kara:05} by equation (9). For any density $\varphi$, the field $w\shdeq S_{r\alpha}\varphi$ solves~\eqref{eq:propag} in $\Rbb^3\shsetm\G_{\alpha}$.

\noindent
Towards the identification of $B$, the medium is excited by source densities $g\in H^{-1/2}(\Gs)$, creating incident fields $u$ that are given by single-layer potentials
\[
  u(\bfx) = S_{rs}g(\bfx), \qquad \bfx\shin\Rbb^3.
\]
In the perturbed medium, this excitation gives rise to the total field $u^g_B$ such that
\[
  -\Div\lpar\bfA_B\sip\bfna u^g_B\rpar - \kappa^2 u^g_B = g\,\delta_{\Gs} \qquad \text{and radiation condition}
\]
(here, since the incident field is radiating, the total field is radiating too). By linear superposition, we have
\[
  u^g_B(\bfx) = \iGs u_B(\bfx;\bfs) g(\bfs) \dbfs \qquad \bfx\shin\Rbb^3
\]
where $u_B(\cdot;\bfs)$ solves 
\begin{equation}
  -\Div\lpar\bfA_B\sip\bfna u_B\rpar - \kappa^2 u_B = \delta(\cdot-\bfs) \qquad \text{and radiation condition.} \label{uB(m;s)}
\end{equation}

\noindent
In the present framework (where sources and measurements are not assumed to be in the far field), point sources and their superposition as potentials replace plane waves and their superposition as Herglotz wave functions used in e.g.~\cite{B-2013-01}.

\subsection{Cost functional} We assume the knowledge on $\Gm$ of a measurement of $u\obs=u\obs(\cdot;\bfs)$ of the field $u_B(\cdot;\bfs)$ for each source location $\bfs\shin \Gs$ and formulate the problem of identifying $B$ in terms of the minimization of a cost functional. Letting $D$ denote the support of a trial inhomogeneity, the least-squares cost functional
\begin{equation}
  \Jcal(D) := \demi \iGs \iGm \labs u_D(\bfm;\bfs)-u\obs(\bfm;\bfs) \rabs^2 \dbfm \dbfs, \label{J:def}
\end{equation}
is the most common basis for such optimization-based identification. For reasons that will appear later, we will consider the modified form
\begin{equation}
  \Jcal_E(D) := \demi \iGs \iGs \labs \lpar Eu_D(\bfs';\bfs)-Eu\obs(\bfs';\bfs) \rpar \rabs^2 \dbfs\dbfs' \label{Js:def}
\end{equation}
of the cost functional~\eqref{J:def}, where $E:H^{1/2}(\Gm)\to H^{1/2}(\Gs)$ is a bounded linear operator (to be specified later) which produces an ``equivalent measurement'' $Eu\obs$ and its model prediction $Eu_D$ that are defined on the source surface $\Gs$ (so $E$ acts on the first variable of the two-point functions $u_D,u\obs$). Moreover, to facilitate the theoretical analysis that follows, we idealize the situation further by assuming the data to be noise-free, i.e. $u\obs(\cdot;\bfs)=u_B(\cdot;\bfs)$. For computations of topological derivative inversion with noisy data in the far-field case for the isotropic media see e.g. \cite{B-2013-01,B-2005-17}.

\subsection{Asymptotic of the cost functional} In this approach, the medium is ``sampled'' by means of trial inhomogeneities $B_{\delta}(\bfz)$ of support $B_{\delta}(\bfz)\sheq\bfz+\delta\Bcal$ and size $\delta>0$, centered at a given point $\bfz\shin\OO$ and endowed with specified material constants $\bfA_z$. Without loss of generality, $\bfz$ can be chosen as the center of $B_{\delta}$, i.e. such that
\begin{equation}
  \iBcal \bfx \dVx = \mathbf{0}. \label{center}
\end{equation}
We then set $D=B_{\delta}=B_{\delta}(\bfz)$ in the cost functional~\eqref{J:def}. Denoting by $u_{\delta}:=u_{B_{\delta}}$ the total field arising in this situation and remembering the error-free assumption made for the measurement, we then define the cost function $J(\delta)=J(\delta;\bfz)$ in terms of $\Jcal$ by
\begin{equation}
  J(\delta) = \Jcal_E(B_{\delta}) = \demi \iGs \iGs \labs Eu_{\delta}(\bfs';\bfs)-Eu_B(\bfs';\bfs) \rabs^2 \dbfs\dbfs' \label{Jeps:def}
\end{equation}
The \emph{topological derivative} $\Tcal(\bfz)$ of $J$ at $\bfz$ is then defined as the leading coefficient in the following expansion of $J(\delta) \shm J(0)$ in powers of $\delta$:
\begin{equation}
  J(\delta) = J(0) + \delta^3\Tcal(\bfz) + o(\delta^3). \label{T:def}
\end{equation}
In view of~\eqref{Jeps:def} and~\eqref{T:def}, the topological derivative $\Tcal(\bfz)$ can be evaluated by identification from \cite{B-2013-01,B-2005-17}:
\begin{equation}
  -\Re \Lcb \iGs\iGs \overline{Eu_{\delta}\ssup(\bfs';\bfs)} \, E u_B\ssup(\bfs';\bfs) \dbfs' \dbfs \Rcb = \delta^3\Tcal(\bfz) + o(\delta^3), \label{T:def2}
\end{equation}
where $u_B\ssup:=u_B\shm u$ and $u_{\delta}\ssup:=u_{\delta}\shm u$ are the scattered fields associated with $u_B$ and $u_{\delta}$, respectively.\enlargethispage*{1ex}

\section{Explicit expression of the topological derivative}

We now have all the ingredients to develop from~\eqref{T:def2} an expression of the topological derivative $\Tcal(\bfz)$ that is convenient for its analysis as an identification tool.

\subsection{Representation of scattered fields}

Recalling known results on the solution's asymptotics (which incidentally explain the expected $O(\delta^3)$ leading order in~\eqref{T:def2}, see e.g. \cite{agjk:11,col:83,B-2005-17}),  and given our choice of incident fields, the scattered field for the trial inhomogeneity $B_{\delta}$ is given at any $\bfx\shneq\bfz$ by the expansion
\begin{equation}
  u_{\delta}\ssup(\bfx;\bfs)
 = \delta^3 W(\bfx;\bfs) + o(\delta^3), \qquad
 W(\bfx;\bfs) := \bfna \Phik(\bfx\shm\bfz)\sip\bfM_z\sip\bfna \Phik(\bfz\shm\bfs), \label{exp:ur}
\end{equation}
where $\bfM_z:=\bfM(\Bcal,\bfA_z)\in\Rbb\sym^{3\times3}$ denotes the polarization tensor of the normalized trial inhomogeneity~\cite{agjk:11,ced:mos:vog}.

\noindent
Moreover, the scattered field for the true inhomogeneity has the representation
\begin{equation}
  u_B\ssup(\bfx;\bfs) = \bfW_{\kappa}[\bfh](\bfx), \label{exp:uB:W}
\end{equation}
where  $\bfW_{\kappa}$ is the volume potential defined for any density $\bfg\in L^2\comp(\Rbb^3;\Cbb^3)$ by
\begin{equation}
  \bfW_{\kappa}[\bfg](\bfx) = \int_{\Rbb^3} \bfna \Phik(\bfx\shm\bfy)\sip\bfg(\bfy) \dbfy \label{W:def}
\end{equation}
and the density $\bfh=(\bfAT\shm\bfA)\sip\bfna u_B(\cdot,\bfs)\in L^2(B;\Cbb^3)$ solves the singular volume integral equation (VIE)
\begin{equation}
  \bfT\bfh = (\bfAT\shm\bfA)\sip\bfna u \quad  \text{in $B$}, \qquad \text{with \ } \bfT := \bfI-(\bfAT\shm\bfA)\sip\bfna\bfW_{\kappa} \label{h:VIE}
\end{equation}
(note that $\text{supp}(\bfAT\shm\bfA)=\overline{B}$). The singular integral operator $\bfT:L^2(B;\Cbb^3)\to L^2(B;\Cbb^3)$ is known to be invertible with bounded inverse. Solving equation~\eqref{h:VIE}, using~\eqref{exp:uB:W} and recalling the definition of $u$, we obtain
\begin{equation}
  u_B\ssup(\bfx;\bfs) = \iB \bfna \Phik(\bfx\shm\bfy)\sip\lsqb\bfM_B\bfna \Phik(\cdot\shm\bfs)\rsqb(\bfy) \dbfy \label{exp:uB}
\end{equation}
with the solution operator $\bfM_B$ defined for any $\bfg\shin L^2(B;\Cbb^3)$ by $\bfM_B\bfg:=\bfh$ with $\bfh$ solving  $\bfT\bfh=(\bfAT\shm\bfA)\cdot\bfg$. We refer the reader to \cite{B-2016-06} for more details on how these expressions are obtained.

\subsection{Source-to-measurement operators and their factorization}

Let the measurement operators $F_B$ and $F_z$ associated with the true and trial scattered fields be defined such that $\gamma_m u_B=F_Bg$ and $\gamma_mu_{\delta}=F_zg + o(\delta^3)$, where $\gamma_m$ denotes the Dirichlet trace operator on $\Gm$ and $g\shin H^{-1/2}(\Gs)$ is any excitation applied on $\Gs$. In view of representations~\eqref{exp:ur} and~\eqref{exp:uB}, we have
\begin{equation}
  F_z = \overline{\bfH}{}^{\star}_{zm} \lpar \delta^3\bfM_B \rpar \bfH_{zs} + o(\delta^3), \qquad F_B = \overline{\bfH}{}^{\star}_{Bm} \bfM_B \bfH_{Bs} \label{nsym:F}
\end{equation}
where the operators $\bfH_{z\alpha}:H^{-1/2}(\G_{\alpha})\to\Cbb^3$ and $\bfH_{B\alpha}:H^{-1/2}(\G_{\alpha})\to L^2(B;\Cbb^3)$ are defined by
\[
  \bfH_{z\alpha}\varphi = \bfna S_{r\alpha}\varphi(\bfz), \qquad \bfH_{B\alpha}\varphi
  = \bfna S_{r\alpha}\varphi\mid_B =: \bfna S_{B\alpha}\varphi \qquad \alpha\sheq m,s.
\]
in terms of the single-layer potential operator~\eqref{S:def}. Here  $ \bfH_{z\alpha}^\star: {\mathbb C}^3\to H^{1/2}(\G_{\alpha})$ and $\bfH_{B\alpha}^\star:L^2(B;\Cbb^3)\to H^{1/2}(\G_{\alpha})$ denote the conjugate transpose which we will refer to as adjoint (note that the duality pairing $H^{-1/2},H^{1/2}$ is with respect to the $L^2$ pivot space). The measurement operators $F_B,F_z$ are thus expressed by~\eqref{nsym:F} as \emph{non-symmetric} factorizations, a feature previously noticed in e.g.~\cite{audibert,audibert:haddar:17,hu:14}. Following~\cite{audibert}, symmetric factorizations can be obtained with the help of the following lemma:
\begin{lemma}\label{E:lemma}
Assume that $\kappa^2$ is not a Dirichlet eigenvalue for the Laplace operator in $R$. If the source/measurement configuration is such that either $\Gm\sheq\Gs$ or $\Gm\shsubs\Rs$ (cases (i) and (ii) of the Introduction), we have
\[
  S^{\star}_{ms}S^{-1}_{mm}\overline{\bfH}{}^{\star}_{Bm} = \bfH^{\star}_{Bs}.
\]
where $S_{ms}:=\gamma_{m}S_{rs}$ while $S_{mm}\shdeq\gamma_{m}S_{rm}$ is the single-layer integral operator on $\Gm$.
\end{lemma}
\begin{proof}
We proceed by proving the (equivalent) adjoint equality $\overline{\bfH}{}_{Bm}(S^{\star}_{mm})^{-1}S_{ms}=\bfH^{\star}_{Bs}$. This equality also reads $\bfna\overline{S}{}_{Bm}\overline{S}^{-1}_{mm}S_{ms}=\bfna S^{\star}_{Bs}$ in view of the definition of $\bfH_{B\alpha}$ and since $S^{\star}_{mm}=\overline{S}_{mm}$. For any given density $\psi_s\shin H^{-1/2}(\Gs)$, $\overline{S}{}_{Bm}\overline{S}^{-1}_{mm}S_{ms}\psi_s$ and $S^{\star}_{Bs}\psi_s$ are Helmholtz solutions in $\Rm\shsubs\Rs$ and $\Rs$, respectively. Taking the trace on $\Gm$ for both fields, we obtain $\gamma_m\overline{S}{}_{Bm}\overline{S}^{-1}_{mm}S_{ms}\psi_s=S_{ms}\psi_s=\gamma_m S^{\star}_{Bs}\psi_s$. Hence the two Helmholtz solutions, having the same trace on $\Gm$, coincide in $\Rm$. Their gradients therefore also coincide in $\Rm$, and the lemma follows by taking the adjoint.

\noindent
Note that due to the lemma assumptions $\overline{S}{}_{Bm}\overline{S}^{-1}_{mm}S_{ms}\psi_s$ is not a Helmholtz solution outside $\Rm$, since $\Rs$ contains the surface $\Gm$ supporting the density $\overline{S}^{-1}_{mm}S_{ms}\psi_s$. This is the reason for the our assumptions.
\end{proof}
\noindent
Therefore, defining the linear bounded operator $E:=S_{ms}^{\star}S^{-1}_{mm}$ from $H^{1/2}(\Gm)$ to $H^{1/2}(\Gs)$ and recalling factorizations~\eqref{nsym:F}, Lemma~\ref{E:lemma} implies the symmetric factorizations
\[
  EF_z = \bfH^{\star}_{zs} \bfM_B \bfH_{zs}, \qquad EF_B = \bfH^{\star}_{Bs} \bfM_B \bfH_{Bs}
\]
or, equivalently:
\begin{align}
  E u_B\ssup(\bfs';\bfs)
 &= \iB \overline{\bfna \Phik(\bfs'\shm\bfy)}\sip\lsqb\bfM_B\bfna \Phik(\cdot\shm\bfs)\rsqb(\bfy) \dbfy, \\
  Eu_{\delta}\ssup(\bfs';\bfs)
 &= \delta^3 \overline{\bfna \Phik(\bfs'\shm\bfz)}\sip\bfM_z\sip\bfna \Phik(\bfz\shm\bfs) + o(\delta^3).
\end{align}
\noindent
For an explicit example of the symmetry restoring-operator $E$, see \ref{seca3}.

\medskip
\noindent
We are finally ready to give for the topological derivative an explicit expression, which is the main object of study in what follows.

\subsection{Topological derivative}
Inserting the above expressions of $E u_B\ssup$ and $Eu_{\delta}\ssup$ in~\eqref{T:def2}, the topological derivative is found to be given by the formula
\begin{multline}
  \Tcal(\bfz)
 = -\Re \Lcb \iGs\iGs\iB \bfna \Phik(\bfs'\shm\bfz)\sip\bfM_z\sip\overline{\bfna \Phik(\bfz\shm\bfs)} \\
  \overline{\bfna \Phik(\bfs'\shm\bfy)}\sip\lsqb\bfM_B\bfna \Phik(\cdot\shm\bfs)\rsqb(\bfy) \ \dbfy \dbfs' \dbfs \Rcb, \label{TD:exp:2}
\end{multline}
which will serve as the main basis for our analysis. This formula can be recast in a more concise, and structure-revealing, form as
\begin{align}
  \Tcal(\bfz)
 &= -\Re\, \Lcb \iB \overline{\bfG(\bfz,\bfy)}\dip \lsqb\MS\bfG \rsqb(\bfz,\bfy) \dbfy \Rcb \nnum
 &= -\Re\, \Lcb \lpar \bfG,\MS\bfG \rpar_{L^2(B;\Cbb^{3\times 3})} \Rcb \label{TD:exp}
\end{align}
where $f,g\mapsto\lpar f,g \rpar$ denotes the sesquilinear form associated with the $L^2(B)$ scalar product (for scalar- or tensor-valued functions as needed), the (two-point, tensor-valued) function $\bfG$ is defined by
\begin{align}
  \bfG(\bfz,\bfy)
 &:= \iGs \overline{\bfna \Phik(\bfs\shm\bfz)}\tens\bfna \Phik(\bfs\shm\bfy) \dbfs, \label{K:def} \\
 \text{i.e. \ }
  G_{ij}(\bfz,\bfy) &= \iGs \overline{\partial_i \Phik(\bfs\shm\bfz)}\;\partial_j \Phik(\bfs\shm\bfy) \dbfs, \notag
\end{align}
and $\MS$ is the $L^2(B;\Cbb^{3\times 3})\to L^2(B;\Cbb^{3\times 3})$ operator given by
\begin{equation}
  \Mcal_{ijk\ell} := (\bfM_z)_{ik}(\bfM_B)_{j\ell}
\end{equation}
\noindent
In addition, $\bfG(\bfz,\bfx)$ is alternatively given by
\begin{equation}
  G_{ij}(\bfz,\bfy) = \derpq{}{z_i}{y_j} L(\bfz,\bfy), \label{G=D2L}
\end{equation}
with the two-point function $L$ defined by
\begin{equation}
  L(\bfz,\bfy) := \iG \overline{\Phik(\bfs\shm\bfz)} \, \Phik(\bfs\shm\bfy) \dbfs. \label{L:def}
\end{equation}
The function $L$ would moreover appear in the counterpart of~\eqref{TD:exp} associated with inhomogeneities characterized solely by a contrast in their refraction index, studied in~\cite{B-2013-01}.
\begin{remark}
{\em The above expressions are also valid if the scattering problem is formulated in a bounded region instead of the entire space. In this case $\Phik(\cdot,\cdot)$ denotes the fundamental solution of the (bounded) background medium satisfying the relevant homogeneous boundary condition.}
\end{remark}

\subsection{Reversed nesting of source/measurement surfaces}
\label{nested}

Lemma~\ref{E:lemma} requires $\Gs$ to surround, or coincide with, $\Gm$. The following reciprocity property  allows to include the case $\Gs\shsubs\Rm$ (i.e. $\Gm$ surrounding $\Gs$, case (iii) of Introduction) in our analysis:
\begin{lemma}\label{uB:reciprocity}
For any inhomogeneity $B$ and any $\bfm,\bfs\shin\Rbb^3$ such that $\bfm\shneq\bfs$ and $\bfm,\bfs\not\in\overline{B}$, the function $u_B(\cdot;\bfs)$ defined by problem~\eqref{uB(m;s)} satisfies $u_B(\bfm;\bfs)=u_B(\bfs;\bfm)$.
\end{lemma}
\begin{proof}
Let $\OO_R$ denote the ball of radius $R$, with $R$ large enough to have $\bfm,\bfs\shin\OO_R$, and set $\OO_{R,\eps}(\bfs):=\lcb \bfx\shin\OO_R, |\bfx\shm\bfs|>\eps \rcb$ with $\eps<|\bfm\shm\bfs|$. We have
\[
  -\int_{\OO_{R,\eps}(\bfs)} \lsqb \Div\lpar\bfA_B\sip\bfna u_B(\cdot;\bfs) \rpar + \kappa^2 u_B(\cdot;\bfs) \rsqb u_B(\cdot;\bfm) \dV = 0
\]
and the above integral is well-defined since $u_B(\cdot;\bfs)$ is smooth, and $u_B(\cdot;\bfm)$ summable, in $\OO_{R,\eps}(\bfs)$. Applying the first Green identity to the above identity and taking the limit $\eps\to0$ in the resulting equality (using that $u_B(\cdot;\bfs)=\Phi_{\kappa}(\cdot-\bfs)+u^s_B(\cdot;\bfs)$ together with the smoothness of $u^s_B(\cdot;\bfs)$ in a neighborhood of $\bfs$) yields
\begin{multline}
  \int_{\OO_{R,\eps}(\bfs)} \lsqb \bfna u_B(\cdot;\bfs) \sip\bfA_B\sip\bfna u_B(\cdot;\bfm) - \kappa^2 u_B(\cdot;\bfs) u_B(\cdot;\bfm) \rsqb \dV \\[-1ex]
 = u_B(\bfs;\bfm) + \int_{\partial\OO_R} \lpar \bfn\sip\bfA\sip\bfna u_B(\cdot;\bfs) \rpar u_B(\cdot;\bfm) \dS.
\end{multline}
The above equality also holds with the roles of $\bfm$ and $\bfs$ reversed. Subtracting these two equalities provides
\begin{multline}
  0 = u_B(\bfs;\bfm) - u_B(\bfm;\bfs) \\[-0.5ex]
  + \int_{\partial\OO_R} \lcb \lpar \bfn\sip\bfA\sip\bfna u_B(\cdot;\bfs) \rpar u_B(\cdot;\bfm) - \lpar \bfn\sip\bfA\sip\bfna u_B(\cdot;\bfm) \rpar u_B(\cdot;\bfs) \rcb \dS.
\end{multline}
The lemma finally follows from the fact that the above integral over $\partial\OO_R$ vanishes in the limit $R\to\infty$ due to the radiation condition~\eqref{rad:cond} satisfied by both $u_B(\cdot;\bfs)$ and $u_B(\cdot;\bfm)$.
\end{proof}

\noindent
Lemma~\ref{uB:reciprocity} implies that the measurement residuals (assuming noise-free data) verify
\[
  u_D(\bfm;\bfs)-u\obs(\bfm;\bfs) = u_D(\bfs;\bfm)-u_B(\bfs;\bfm), \qquad \bfm\shin\Gm,\,\bfs\shin\Gs
\]
Consequently, when $\Gm$ surrounds $\Gs$, the foregoing analysis leading to~\eqref{TD:exp} still applies by the simple expedient of reversing the roles of $\Gs$ and $\Gm$ in the cost functionals~\eqref{J:def} and~\eqref{Js:def} and setting $E:=S^{\star}_{sm}V^{-1}_{ss}$ for the symmetry-restoring operator $E$. Accordingly, the topological derivative is in this case given by
\begin{multline}
  \Tcal(\bfz)
 = -\Re \Lcb \iGm\iGm\iB \bfna \Phik(\bfm'\shm\bfz)\sip\bfM_z\sip\overline{\bfna \Phik(\bfz\shm\bfm)} \\
  \overline{\bfna \Phik(\bfm'\shm\bfy)}\sip\lsqb\bfM_B\bfna \Phik(\cdot\shm\bfm)\rsqb(\bfy) \ \dbfy \dbfm' \dbfm \Rcb, \label{TD:exp:reversed}
\end{multline}
Now we are ready to study the behavior of $ \Tcal(\bfz)$ for various locations of sampling point $z$. We will begin, in Section \ref{seciso1}, with the simpler case of isotropic media.

\subsection{Cases of partial aperture}

The foregoing development, which is undertaken assuming both surfaces $\Gs$ and $\Gm$ to be closed (and either nested or equal), can be extended to the cases where the outside surface is open, i.e. either $\Gs$ is open, $\Gm$ is closed and $\Gs\subset(\Rbb^3\shsetm\Rm)$ or $\Gm$ is open, $\Gs$ is closed and $\Gm\subset(\Rbb^3\shsetm\Rs)$. In the former case, Lemma~\ref{E:lemma} still holds true, with its proof unchanged except for the fact that the image space in identity $S^{\star}_{ms}V^{-1}_{mm}\overline{\bfH}{}^{\star}_{Bm} \sheq \bfH^{\star}_{Bs}$ is $H^{1/2}(\Gs)$, which requires that both members of the adjoint equality be evaluated on densities $\psi_s\shin \widetilde{H}{}^{-1/2}(\Gs)$. 
Hence the symmetry-restoring operator $E$ remains defined by $E:=S_{ms}^{\star}S^{-1}_{mm}$, and the resulting expression~\eqref{TD:exp} still holds. In the latter case, the reciprocity Lemma~\ref{uB:reciprocity} again allows reversion to the former case as explained in Section ~\ref{nested}.

\section{Isotropic scatterers}\label{seciso1} In this case, we have $\bfA=a\bfI$, $\bfAT=\tila\bfI$, $\bfA_z=a_z\bfI$, where $a$, $\tila$ and $a_z$ are strictly positive material constants. We introduce for convenience the non-dimensional material parameters
\begin{equation}\label{koti}
  \beta:=\frac{\tila}{a}\shm1, \quad  \beta_z:=\frac{a_z}{a}\shm1, \qquad
  q:=\frac{\beta}{\beta\shp2}, \quad q_z = \frac{\beta_z}{\beta_z\shp2},
\end{equation}
which verify $-1\shl\beta,\beta_z\shl\infty$ and $-1\shl q,q_z\shl1$. 
As we will see in the following, for isotropic scatterers the topological derivative expression is easier to analyze.
\subsection{Simplified expression of the topological derivative}

The singular integral operator $\bfT$ introduced in~\eqref{h:VIE} is then given by
\[
  \bfT = \bfI - a\beta\bfna\bfW_{\kappa} = \frac{\beta}{2q}(\bfI-q\bfR_{\kappa}), \qquad \text{with}\quad \bfR_{\kappa}:=\bfI\shp 2a\bfna\bfW_{\kappa}
\]
(with the second equality easily checked by inspection). The solution operator $\bfM_B$ introduced in~\eqref{exp:uB} is then given by
\[
  \bfM_B = 2aq(\bfI-q\bfR_{\kappa})^{-1}.
\]
Moreover, the polarization tensor, being defined from the zero-frequency transmission problem where $\Bcal$ is excited by a remote constant gradient, is given by
\[
  \bfM_z\sip\bfg = 2aq_z \iBcal \lpar \bfI - q_z\bfR_0 \rpar^{-1}\bfg \dV \qquad \text{for any $\bfg\shin\Cbb^3$}
\]
with $\bfR_0:=\bfI\shp 2a\bfna\bfW_0$, and where the volume potential $\bfW_0$ is defined as in~\eqref{W:def} except that $\Phik$ is replaced with the \emph{zero-frequency} fundamental solution $\Phiz$, given by $\Phiz(\bfr) = 1/(4\pi a|\bfr|)$. Since $\|q_z\bfR_0\|<1$ for any $q_z>-1$ ~\cite{B-2016-06} and $\bfR_0$ defines a real symmetric $L^2(\Bcal;\Rbb^3)\to L^2(\Bcal;\Rbb^3)$ operator, the operator $\bfI - q_z\bfR_0$ is symmetric and positive definite, implying that the polarization tensor can be recast  in the form
\[
  \bfM_z = 2aq_z \bfD_z\Tsup\sip\bfD_z
\]
(with $\bfD_z$ the real-valued Choleski square root of the real symmetric positive definite matrix $(2aq_z)^{-1}\bfM_z$). If the trial inhomogeneity is spherical (i.e. if $\Bcal$ is the unit ball), we have
\begin{equation}
  \bfM_z = \frac{4\pi a\beta_z}{\beta_z\shp 3}\bfI = \frac{8\pi aq_z}{3\shm q_z}\bfI, \quad\text{i.e.}\quad
  \bfD_z = \sqrt{\frac{4\pi}{3\shm q_z}}\bfI. \label{Mz:iso:ball}
\end{equation}

\noindent
We now take advantage of the above representation of $\bfM_z$ in the expression~\eqref{TD:exp} of $\Tcal(\bfz)$, which becomes
\begin{equation}
  \Tcal(\bfz)
 = -4a^2 qq_z\, \Re\, \Lcb \lpar \bfK,\RS\bfK \rpar_{L^2(B;\Cbb^{3\times 3})} \Rcb \label{TD:exp:modif}
\end{equation}
with $\bfK(\bfz,\bfy):=\bfD_z\Tsup\sip\bfG(\bfz,\bfy)$ and
\begin{equation}
  \Rcal_{ijk\ell} := \delta_{ik}\lpar \bfI-q\bfR_{\kappa} \rpar^{-1}_{j\ell}. \label{M:def}
\end{equation}
As a result of ~\eqref{TD:exp:modif} and~\eqref{M:def}, we can deduce that, under an assumption on the strength of the scatterer, the sign heuristic underpinning topological derivative-based identification is true. More specifically, with the stated notations and assumptions on the scattering by isotropic media with contrast in the main operator (as opposed to \cite{B-2013-01} where the contrast is only in the lower order term), we have proven the following theorem.
\begin{theorem}\label{th-iso}
For any true isotropic scatterer $(B,\beta)$, where $\beta$ is  defined by (\ref{koti}), and wave number $\kappa$ that satisfy
\begin{equation}
  \|q\bfR_{\kappa}\| = |q|\,\|\bfR_{\kappa}\|<1, \label{norm:iso}
\end{equation}
the topological derivative satisfies the sign condition
\begin{equation}
  \text{\em sign}(\Tcal(\bfz)) = -\text{\em sign}(qq_z), \label{sign:iso}
\end{equation}
where $q$ and $q_z$ are given by (\ref{koti}).
\end{theorem}

\noindent
Condition~\eqref{norm:iso} can be considered as restricting the justification of the sign heuristic to ``moderate'' scatterers (the moderate character depending on a combination of the scatterer size, its material contrast and the operating frequency).  We call  the scatterers that satisfies condition~\eqref{norm:iso} moderate, since it is less restrictive than the weak scattering condition implicit in the Born approximation (see Sec.~\ref{Born})

\noindent
As discussed in \cite{B-2013-01}, to use $z\mapsto \Tcal(\bfz)$ as an identifying function for the inhomogeneity, it should decay as $\bfz$ moves far away from the boundary of the unknown inhomogeneity in addition to verifying the sign heuristic property. But as opposed to \cite{B-2013-01}, here we deal with near field data and hence we need to understand how $\Tcal(\bfz)$ decays for $\bfz$ ``far" from the boundary of the inhomogeneity $B$ and still remaining within a ``reasonable" distance from the measurement curve $\Gamma_m$. To address this issue,  next we  carry out  this two-scale asymptotic calculations  for a spherical configuration of the measurement/source surface.

\subsection{Decay properties of the topological derivative}

Here we limit ourselves to the case when the trial inhomogeneity is spherical (i.e. if $\Bcal$ is the unit ball) and when the excitations and measurements surfaces $\Gamma_s=\Gamma_m=R\hat S$ are  both the sphere of radius $R$ centered at the origin. For the purpose of these calculations, we assume without loss of generality that $a=1$, hence $\Phik$ defined by~\eqref{fund:gov:eq} is now the free space fundamental solution of the Helmholz equation given by
\begin{equation}\label{fund} 
\Phik(\bfs\shm\bfy):=\frac{1}{4\pi}\frac{e^{\rmi\kappa|\bfs\shm\bfy|}}{|\bfs\shm\bfy|}.
\end{equation}
 In this particular setting, as noted above, the topological derivative becomes
\begin{equation}
  \Tcal(\bfz)
 = -\frac{16\pi qq_z}{3-q_z}\, \Re\, \int_B \bfG(\bfz,\bfy) : \left[\RS \bfG\right](\bfz,\bfy) \dbfy \label{TD:exp:modif2}
\end{equation}
where 
$$\Rcal_{ijk\ell} := \delta_{ik}\lpar \bfI-q\bfR_{\kappa} \rpar^{-1}_{j\ell}, \;  \bfR_{\kappa}:=\bfI\shp 2a\bfna\bfW_{\kappa}, \; \bfW_{\kappa}[\bfg](\bfx) = \int_{B} \bfna \Phik(\bfx\shm\bfy)\sip\bfg(\bfy) \dbfy$$
and the $3\times 3$ tensor valued function $\bfG(\bfz,\bfx)$ is given by\enlargethispage*{1ex}
\begin{multline}
  \bfG(\bfz,\bfy)=\int_{R\hat S}\bfna_z\overline{\Phik(\bfs\shm\bfz)} \otimes \bfna_y \Phik(\bfs\shm\bfy) \dbfs \\
 = \frac{1}{16\pi^2}\int_{\hat S}\frac{(1+\rmi\kappa|\bfs\shm\bfz|)}{|\bfs\shm\bfz|^2}\frac{(1-\rmi\kappa|\bfs\shm\bfy|)}{|\bfs\shm\bfy|^2}e^{-\rmi\kappa|\bfs\shm\bfz|}e^{\rmi\kappa|\bfs\shm\bfy|}(\widehat{\bfs\shm\bfz}\otimes \widehat{\bfs\shm\bfy})\, R^2 d \hat \bfs. \label{eqg}
\end{multline}
\begin{figure}[b]
\begin{center}
\includegraphics[width=0.40\textwidth]{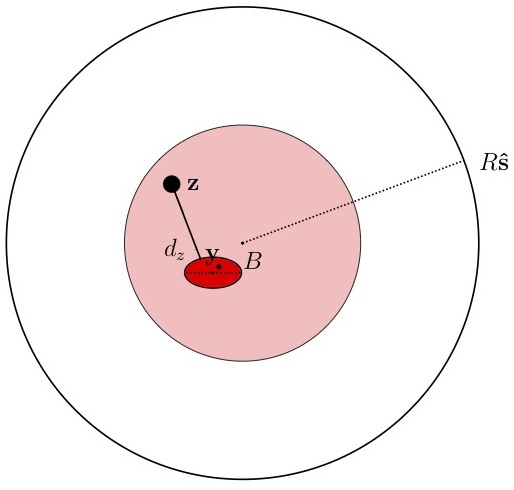} 
\end{center}
\caption{A sketch of the probing region. The thick line, i.e. $d_z$, indicates the reference length scale, which is much smaller than $R$, more precisely $d_z/R=\eta^\alpha$, but much bigger then diam$(B)$,  more precisely diam$(B)/d_z=\eta^{1-\alpha}$. Here  $\eta=\mbox{diam}(B)/R$.}
\label{Fig1}
\end{figure}
We want to study the decaying behavior of   $\Tcal(\bfz)$ for $\bfz$ far away from the target inhomogeneity $B$. In the far field it was shown in \cite{B-2013-01} (for a slightly different problem) that the topological derivative decays at reversed proportional to the square of the distance  of $\bfz$ for $B$. However, here we expect that such behavior depends on  how far the probing region is from the source/measurement surface. To better understand this interplay, for a fixed $\bfz$ outside $B$ we set the reference length to be $d_{\bfz}:=\mbox{dist}(\bfz, B)$  and note that $|\bfy\shm\bfz|$, $\bfy\in B$ is $O(d_{\bfz})$.  Let $\eta>0$ be a small parameter and take a constant $0<\alpha<1$. Here $\eta$ characterizes the ratio between the size of $B$ and the radius $R$ of the measurement/source sphere (Figure~\ref{Fig1}). Thus
\begin{equation}\label{a0}
\frac{|\bfy|}{|\bfs|}=\frac{|\bfy|}{R}=O(\eta), \qquad \bfy\in B,\; \bfs \in R\hat S.
\end{equation}
We express the facts that  the ``region of action" (i.e. the probing region and inhomogeneity) is far from the source/measurement surface, and that  $\bfz$ stays ``far from" the inhomogeneity, by assuming that 
\begin{equation}
\frac{|\bfy\shm\bfz|}{|\bfs|}=\frac{|\bfy\shm\bfz|}{R}=O(\eta^{\alpha}) \qquad \mbox{and} \qquad \frac{|\bfy|}{|\bfy\shm\bfz|}=O(\eta^{1-\alpha}), \label{a1}
\end{equation} 
respectively, uniformly for ${\bfy}\in B$.  Loosely speaking our scaling is such that $d_{\bfz}/R=\eta^\alpha$ and $\mbox{diam}(B)/d_{\bfz}=\eta^{1-\alpha}$ (see Figure~\ref{Fig1}). We now perform ``far field"  asymptotic expansions for functions involved in (\ref{eqg}) as $\eta\to 0$, retaining only the terms of order $O(1)$ and $O(\eta^\alpha)$ (note that the terms $O(1)$ are those that appear in the far field expansion \cite{B-2016-06}). To this end, making  use of the following simple formula
\begin{equation}
  |\bfs\shm\bfz|-|\bfs\shm\bfy|=\frac{|\bfs\shm\bfz|^2-|\bfs\shm\bfy|^2}{|\bfs\shm\bfz|+|\bfs\shm\bfy|}=|\bfy\shm\bfz|\frac{|\bfy\shm\bfz|+2(\bfs\shm\bfy)\cdot \widehat{(\bfy\shm\bfz)}}{|\bfs\shm\bfz|+|\bfs\shm\bfy|},
\end{equation}
letting
\begin{equation}
c  :=\hat \bfs \cdot \widehat{(\bfy\shm\bfz)},
\end{equation}
and noting that
\begin{equation}
  \frac{\bfs\shm\bfy}{R}=\hat \bfs(1+O(\eta))=\bfsh(1+o(\eta^\alpha)),
\end{equation}
we arrive at the following asymptotic expressions
\begin{equation}
  \displaystyle{|\bfs\shm\bfy|^2=R^2+O(\eta)=R^2+o(\eta^\alpha)}, \qquad  \displaystyle{|\bfs\shm\bfz|^2=R^2\left[1+2c\frac{|\bfy\shm\bfz|}{R}+o(\eta^\alpha)\right]}
\end{equation}
which yield
\begin{align}
  \frac{1-\rmi\kappa|\bfs\shm\bfy|}{|\bfs\shm\bfy|^2}
 &= \frac{1-\rmi\kappa R}{R^2}\left(1+o(\eta^\alpha)\right) \label{a2} \\
\frac{1+\rmi\kappa|\bfs\shm\bfz|}{|\bfs\shm\bfz|^2}
 &=\frac{1}{R^2}\left(1+\rmi\kappa R-(2+\rmi\kappa R)c\frac{|\bfy\shm\bfz|}{R}\right)+o(\eta^\alpha). \label{a3}
\end{align}
Next, we have
\begin{align}
  |\bfy\shm\bfz|+2(\bfs\shm\bfy)\cdot \widehat{(\bfy\shm\bfz)}
 &= R\left[\frac{|\bfy\shm\bfz|}{R}+2c+o(\eta^\alpha)\right] \\
  |\bfs\shm\bfz|+|\bfs\shm\bfy|
 &= R\left[2+c\frac{|\bfy\shm\bfz|}{R}+o(\eta^\alpha)\right]
\end{align}
which from the above yields the following expansion for the exponents in the exponential terms in (\ref{eqg})
$$-\rmi\kappa\left(|\bfs\shm\bfz|-|\bfs\shm\bfy|\right)=-\rmi\kappa|\bfy\shm\bfz|c\left[1+\frac{1-c^2}{2c}\frac{|\bfy\shm\bfz|}{R}+o(\eta^\alpha)\right].$$
Hence, we obtain the following expression for the exponential term
\begin{equation}\label{a4}
e^{-\rmi\kappa\left(|\bfs\shm\bfz|-|\bfs\shm\bfy|\right)}=e^{-\rmi\kappa c |\bfy\shm\bfz|}\left[1-\rmi\kappa\frac{|\bfy\shm\bfz|^2}{R}\frac{1-c^2}{2}+o(\eta^\alpha)\right]
\end{equation}
Now plugging (\ref{a2}), (\ref{a3}) and (\ref{a4}) in (\ref{eqg}), using
\begin{equation}
  \widehat{\bfs\shm\bfz}\otimes\widehat{\bfs\shm\bfy}=\hat \bfs\otimes \hat \bfs+(\hat \bfs\otimes \hat \bfz-c\bfsh\otimes\bfsh)\frac{|\bfy\shm\bfz|}{R}+o(\eta^\alpha)
\end{equation}
and collecting  the coefficients in front of $O(1)$ and $O(\eta^\alpha)$  terms, we finally obtain  (recall that $|\bfy\shm\bfz|/R=O(\eta^\alpha)$, see~\eqref{a1}),
\begin{equation}\label{g1}
  \bfG(\bfz,\bfy)
 = \inv{16\pi^2}\int_{\hat S} \Lcb A(\bfz,\bfy)\bfsh\otimes\bfsh
  + \frac{|\bfy\shm\bfz|}{R} \lsqb A(\bfz,\bfy)\bfsh\otimes\bfzh
  + B(\bfz,\bfy) \bfsh\otimes \bfsh \rsqb \Rcb d{\bfs}+o(\eta^\alpha)
\end{equation}
 with
\begin{align}
  A(\bfz,\bfy)
 &=\frac{1+\rmi\kappa R}{R}\, \frac{1-\rmi\kappa R}{R}e^{-\rmi\kappa c|\bfz\shm\bfy|} \\
  B(\bfz,\bfy)
 &= \frac{1-\rmi\kappa R}{R}\left[\frac{1+\rmi\kappa R}{R}\left(-\rmi\kappa\frac{1-c^2}{2}|\bfy\shm\bfz|-c\right)-\frac{2+\rmi\kappa R}{R}c\right]
\end{align}
where we recall again that $c:=\hat \bfs \cdot \widehat{(\bfy\shm\bfz)}$.  The integration over the unit sphere $\hat S$ after parametrizing it as $x=\sqrt{1-c^2}\cos\varphi$, $y=\sqrt{1-c^2}\sin\varphi$, $z=c$ with  $c\in[-1,1]$ and $\varphi\in [0, 2\pi]$  involves  integrals of the form
\begin{equation}
  I_k:=\int_{-1}^1 e^{\rmi\kappa|\bfy\shm\bfz| c}c^m\,dc\qquad 0\leq m\leq 4
\end{equation}
which from the Jacobi-Anger expansion can be written as linear combinations of
\begin{equation}
  2(-i)^nj_n(\kappa|\bfy\shm\bfz|)
 :=\int_1^1e^{\rmi\kappa|\bfy\shm\bfz| c}P_n(c)dc, \qquad 0\leq n\leq 4
\end{equation}
where $j_n$ are spherical Bessel function of order $n$ and $P_n$ are the Lagrange polynomials. By straightforward but careful calculations, we arrive at the following  expression after using the classical identity $(j_{n-1}+j_{n+1})(t)=(2n\shp 1)j_n(t)/t$ (see e.g \cite[equation (2.34)]{col:83})
\begin{multline}
  \bfG(\bfz,\bfy)
 = \frac{1+\kappa^2 R^2}{12\pi R^2} \Lsqb j_0(\kappa|\bfy\shm\bfz|)\bfI
  + j_2(\kappa|\bfy\shm\bfz|) (\bfI-3\bfbh\otimes\bfbh) \Rsqb \\ \hspace*{-1em}
  - \frac{\kappa R\shp\rmi}{4\pi R^2}
  \Lsqb j_1(\kappa|\bfy\shm\bfz|)\bfbh\otimes\bfbh
  + (\rmi\kappa R+2)\frac{j_2(\kappa|\bfy\shm\bfz|)}{\kappa|\bfy\shm\bfz|}(\bfI-3\bfhb\otimes\bfbh) \Rsqb\frac{|\bfy\shm\bfz|}{R} + o(\eta^{\alpha}),\hspace*{-1em}
\end{multline}
wherein $\bfbh:=\widehat{\bfy\shm\bfz}$. Notice that, for fixed $R$ large enough with respect to the inhomogeneity $B$, i.e. respecting (\ref{a0}) that ensures our asymptotic works, we see from the behavior of Bessel functions $j_n(x)=O(|x|^{-1})$ for $x\to \infty$, that  $\bfG(\bfz,\bfy)=O(|\bfy\shm\bfz|^{-1})$ for $\bfy\in B$ and $\bfz$ far from $B$ (but still respecting (\ref{a1})). Plugging $\bfG(\bfz,\bfy)$ in~\eqref{TD:exp:modif2} and using the fact that $\|\RS\|<C$ in the operator norm, by invoking the Cauchy-Schwarz inequality, we can assert that the topological derivative  $\Tcal(\bfz)$ decays as $O(d^{-2}_{\bfz})$ for large enough $d_{\bfz}:=\mbox{dist}(\bfz, B)$ (i.e. at the same rate as in the case of the far-field for the problem considered in \cite{B-2013-01}). We summarize our result in the following theorem.
\begin{theorem}\label{theo}
For a given unknown isotropic inhomogeneity  $(B,\beta)$, where $\beta$ is  defined by (\ref{koti}), we assume that  the excitations and measurements surfaces $\Gamma_s=\Gamma_m=R\hat S$ are  both the sphere of radius $R$ centered at the origin. Furthermore,  suppose that  $\mbox{\em dist}(\bfz, B)/R=\eta^\alpha$ and $\mbox{\em diam}(B)/\mbox{\em dist}(\bfz, B)=\eta^{1-\alpha}$ for  some small dimensionless parameter $\eta>0$ and $0<\alpha<1$. Then 
$$\Tcal(\bfz)=O\left(\frac{1}{(\mbox{\em dist}(\bfz, B))^{2}}\right) \qquad \qquad \mbox{as \;\; $\mbox{ \em dist}(\bfz, B)\to \infty$}.$$
\end{theorem}

\noindent
Note that in the above calculations we consider  $\Gamma_s=\Gamma_m=R\hat S$  merely for  convenience. The same asymptotic behavior is valid for more general reasonable excitations and measurements surfaces, for example the boundary of a star shaped domain. 

\begin{remark}
{\em  The decaying property of the topological derivative  $\Tcal(\bfz)$ does not depend on the choice of $\alpha \in (0,1)$ which quantifies the fact that $R$ is much larger than the probing region.  Note that  in~\eqref{g1} the only term that could possibly affect the $\alpha$-independent decaying  of $\bfG(\bfz,\bfy)$ for large  $|\bfy\shm\bfz|$, is the term in $B(\bfz,\bfy)$ containing $|\bfy\shm\bfz|$. In our calculations we paid special attention to it; thanks to recursive formulas for Bessel functions, this term disappears. However, we think that the choice of  $0<\alpha<1$ may play a role  if the scattering problem is considered in a bounded region with prescribed boundary data, in which case the derivation of the topological derivative still holds true with $\Phik(\bfs\shm\bfy)$ in~\eqref{fund} replaced by the Green's function of the bounded region.}
\end{remark}

\begin{remark}[Zero-frequency limit] {\em When $\kappa\sheq0$, we have $\|q\bfR_0\|<1$ for any physically admissible $q$, so that~\eqref{sign:iso} holds for any configuration $B,q$. However, we now also have $|\bfK(\bfz,\bfy)|=O(1)$ (i.e. $\bfK(\bfz,\bfy)$ does not decay as the sampling point $\bfz$ is moved away from $B$), implying that the support of $B$ can no longer be (even roughly) estimated on the basis of the function $\bfz\mapsto\Tcal(\bfz)$.}
\end{remark}

\noindent
In \ref{seca2} it is shown that  $\bfG(\bfz,\bfy)$ is real-valued for the  configuration discussed here.  The above calculations  simplify in the case of the far field limit, i.e. $R\to \infty$, but nevertheless yielding exactly the same decaying property of the topological derivative, see \ref{seca1}

\section{Anisotropic scatterer}

The objective here  is to set up for the more general case of anisotropic media a formula for $\Tcal(\bfz)$ that has the same general structure as~\eqref{TD:exp:modif}, and then use it for deducing results on the sign of the topological derivative.  To recast  the topological derivative  in a form allowing to understand its sign, we need a reformulation of the solution operator. To this end, we recall that in~\cite{B-2016-06}, the solution operator $\bfM_B$ is found to have the representation
\[
  \bfM_B = 2\bfA^{1/2}\sip\lpar \bfI\shm\bfQ\sip\bfR_{\kappa} \rpar^{-1}\sip\bfQ\sip\bfA^{1/2}
\]
where $\bfA^{1/2}$ is the positive square root of the positive definite constitutive matrix $\bfA$, the multiplication operator $\bfQ$ is defined by the matrix
\begin{equation}\label{defb}
  \bfQ = (\bfbe\shp2\bfI)^{-1}\bfbe, \qquad \bfbe := \bfA^{-1/2}\sip(\bfAT\shm\bfA)\sip\bfA^{-1/2}
\end{equation}
in terms of the above-defined anisotropic relative contrast $\bfbe$, and the operator $\bfR_{\kappa}$, which depends only on the background medium, is defined by
\begin{equation}
  \bfR_{\kappa} = \bfI+2\bfA^{1/2}\sip\bfna\bfW_{\kappa}\sip\bfA^{1/2}
\end{equation}
Moreover, there exists a matrix $\bfq\in\Rbb^{3\times3}$ and a diagonal matrix $\bfsig$ such that  $\bfQ$ can be factorized as 
\begin{equation}\label{qfac}
\bfQ=\bfq \Tsup \sip \bfsig^2 \sip \bfq,
\end{equation}
 with the nonzero entries of $\bfsig^2$ (also diagonal) being $\pm1$ according to the sign of the corresponding eigenvalue of $\bfQ$. Using this decomposition of $\bfQ$ in $\bfM_B$, we can show that
\begin{equation}
  \bfM_B = 2(\bfA^{1/2}\sip\bfq\Tsup)\sip \bfsig\sip\lpar \bfI\shm\bfsig\sip\bfq\sip\bfR_{\kappa}\sip\bfq\Tsup\sip\bfsig \rpar^{-1}  \sip\bfsig\sip(\bfq\sip\bfA^{1/2}) \label{MB:sym}
\end{equation}
We next consider two different cases.

\subsection{Isotropic background and trial materials, spherical trial inhomogeneity}

Consider the special (and practically useful) case where an anisotropic inhomogeneity embedded in an isotropic background medium is to be identified on the basis of the topological derivative defined in terms of a spherical trial inhomogeneity whose constitutive material is isotropic. In this case, formula~\eqref{TD:exp} for $\Tcal(\bfz)$ can be written with $\bfM_z$ given by~\eqref{Mz:iso:ball} and $\bfM_B$ given by~\eqref{MB:sym} with $\bfA\sheq a\bfI$, i.e.
\[
  \bfM_B
  = 2a\bfq\Tsup\sip\bfsig\sip\lpar \bfI\shm\bfsig\sip\bfq\sip\bfR_{\kappa}\sip\bfq\Tsup\sip\bfsig \rpar^{-1}  \sip\bfsig\sip\bfq,
\]
and we find
\begin{equation}
\Tcal(\bfz) = -4a^2 q_z\, \Re\, \Lcb \Lpar \bfG\sip\bfq\sip\overline{\bfsig},\RS\lsqb \bfG\sip\bfq\sip\bfsig \rsqb \Rpar_{L^2(B;\Cbb^{3\times 3})} \Rcb \label{TD:aniso:iso}
\end{equation}
with the $L^2(B;\Cbb^{3\times 3})\to L^2(B;\Cbb^{3\times 3})$ operator $\RS$ this time defined by

$$ \Rcal_{ijk\ell} = \delta_{ik}\lpar \bfI\shm\bfsig\sip\bfq\sip\bfR_{\kappa}\sip\bfq\Tsup\sip\bfsig \rpar^{-1}_{j\ell}.$$
 
 \noindent 
We are ready to obtain the resulting properties of the topological derivative. Indeed, if the true anisotropic  refractive index  contrast has a sign (i.e. if $\bfsig\sheq\sigma\bfI$ with $\sigma\sheq1$ or $\sigma\sheq\mathrm{i}$), 
\eqref{TD:aniso:iso} becomes
\[
  \Tcal(\bfz)
 = -4a^2 q_z\sigma^2\, \Re\, \Lcb \lpar \bfK,\RS\bfK \rpar_{L^2(B;\Cbb^{3\times 3})} \Rcb
\]
where  $\bfK(\bfz,\bfy):=\bfG(\bfz,\bfy) \sip \bfq$. Consequently, the sign heuristic is true for any true scatterer $(B,\bfbe)$ and wave number $\kappa$ that satisfy
\begin{equation}
  \|\bfq\sip\bfR_{\kappa}\sip\bfq\| <1.\label{norm:iso:aniso}
\end{equation}
Hence we have the following result:
\begin{theorem}\label{th-iso2}
Given the true anisotropic scatterer $(B,\bfbe)$  with $\bfbe$ defined by (\ref{defb}), we assume that the background is isotropic $\bfA\sheq a\bfI$ and the contrast $\bfAT\shm\bfA$ has a definite sign, in the sense that in the factorization (\ref{qfac}) $\bfsig^2=\sigma^2 \bfI$ with $\sigma^2=\pm 1$. Then, if we consider a spherical isotropic  trial inhomogeneity (i.e. $\Bcal$ the unit ball and $\bfA_z\sheq a_z\bfI$)  and a wave number $\kappa$ such that 
\begin{equation}
   \|\bfq\sip\bfR_{\kappa}\sip\bfq\| <1, \label{norm:iso2}
\end{equation}
the topological derivative satisfies the following sign condition
\begin{equation}
  \text{\em sign}(\Tcal(\bfz)) = -\text{\em sign}(\sigma^2 q_z), \label{sign:iso2}
\end{equation}
where the trial contrast $q_z$ is defined by (\ref{koti}).
\end{theorem}
\noindent
Again here the assumption (\ref{norm:iso2}) can be considered as restricting the justification of the sign heuristic to moderately strong scatterers depending on a combination of the scatterer size, its material contrast and the operating frequency. The one-sign contrast type restriction is not unusual in the justification of  a variety of qualitative methods such as linear sampling and factorization methods.
\subsection{The general anisotropic case}
We now consider the more general case where $\bfA$ and $\bfA_z$ may be anisotropic and the trial inhomogeneity shape $\Bcal$ is arbitrary.
First we conveniently reformulate the polarization tensor. To this end, for the trial inhomogeneity $B_{\delta}$ and its normalized counterpart $\Bcal$, we likewise set
\begin{gather}
  \bfQ_z = (\bfbe_z\shp2\bfI)^{-1}\bfbe_z \qquad\text{with \ } \bfbe_z := \bfA^{-1/2}\sip(\bfA_z\shm\bfA)\sip\bfA^{-1/2}, \\
  \bfR_0 = \bfI+2\bfA^{1/2}\sip\bfna\bfW_0\sip\bfA^{1/2},
\end{gather}
with the zero-frequency fundamental solution $\Phiz$ entering the volume potential $\bfW_0$ now given by
\begin{equation}
  \Phiz(\bfr) = \inv{4\pi\sqrt{\text{det}(\bfA)}} \, \inv{|\bfA^{-1/2}\sip\bfr|}. \label{Phiz:exp}
\end{equation}
Using these definitions, we have
\begin{equation}
  \bfM_z\sip\bfg = \iBcal 2\bfA^{1/2}\sip\lpar \bfI\shm\bfQ_z\sip\bfR_0 \rpar^{-1}\sip\bfQ_z\sip\bfA^{1/2}\bfg \dV \label{M:exp}
\end{equation}
for any $\bfg\shin\Cbb^3$. Therefore, introducing the factorization $\bfQ_z=\bfq_z\Tsup\sip\bfsig_z^2\sip\bfq_z$ of $\bfQ_z$ as in~\eqref{qfac}, an identity similar to~\eqref{MB:sym} holds for $\bfM_z$:
\[
  \bfM_z\sip\bfg = \iBcal 2(\bfA^{1/2}\sip\bfq_z\Tsup)\sip \bfsig_z\sip\lpar \bfI\shm\bfsig_z\sip\bfq_z\sip\bfR_0\sip\bfq_z\Tsup\sip\bfsig_z \rpar^{-1}  \sip\bfsig_z\sip(\bfq_z\sip\bfA^{1/2})\cdot\bfg \dV. \label{Mz:sym}
\]
Besides, the $L^2(B;\Cbb^3)\to L^2(B;\Cbb^3)$ operators $\bfq_z$ and $\bfR_0$ are bounded and verify $\|\bfq_z\|\shl 1$ and $\|\bfR_0\|\sheq 1$~\cite{B-2016-06}. Consequently, the mapping 
$$\bfh\shin\Cbb^3\mapsto\iBcal \lpar \bfI\shm\bfsig_z\sip\bfq_z\sip\bfR_0\sip\bfq_z\Tsup\sip\bfsig_z \rpar^{-1}  \sip\bfh \dV\in\Cbb^3$$ defines a positive definite matrix (hence having a Choleski square root $\bfD_z$) and $\bfM_z$ can be recast as
\begin{equation}
  \bfM_z = 2(\bfA^{1/2}\sip\bfq_z\Tsup)\sip \bfsig_z\sip \bfD_z\Tsup\sip\bfD_z\sip \bfsig_z\sip(\bfq_z\sip\bfA^{1/2}). \label{Mz:sym}
\end{equation}
In this case, formula~\eqref{TD:exp} for $\Tcal(\bfz)$ can be written with $\bfM_z$ given by~\eqref{Mz:iso:ball} and $\bfM_B$ given by~\eqref{MB:sym}, to obtain
\begin{eqnarray}
 &&\hspace*{-2.5cm} \Tcal(\bfz) =\label{TD:aniso} \\
 &&\hspace*{-2.3cm} - \Re\, \Lcb \Lpar  \bfD_z\sip\overline{\bfsig}_z\sip\bfq_z\Tsup\sip\bfA^{1/2}\sip\bfG\sip\bfA^{1/2}\sip\bfq\sip\overline{\bfsig}\,, \,\RS\lsqb \bfD_z\sip\bfsig_z\sip\bfq_z\Tsup\sip\bfA^{1/2}\sip\bfG\sip\bfA^{1/2}\sip\bfq\sip\bfsig \rsqb \Rpar_{L^2(B;\Cbb^{3\times 3})} \Rcb \nonumber
\end{eqnarray}
with the $L^2(B;\Cbb^{3\times 3})\to L^2(B;\Cbb^{3\times 3})$ operator $\RS$ again defined as in~\eqref{TD:aniso:iso}.

\noindent
The above expression allows us to study  the sign  of the topological derivative in some special cases and obtain a result of the type as in Theorem \ref{th-iso2}. More specifically, if both the true and trial anisotropic conductivities have a sign (i.e. if $\bfsig\sheq\sigma\bfI$ and $\bfsig_z\sheq\sigma_z\bfI$ with $\sigma,\sigma_z\sheq1$ or $\mathrm{i}$), 
\eqref{TD:aniso} becomes
\[
  \Tcal(\bfz) = -\sigma^2\sigma_z^2\, \Re\, \Lcb \lpar \bfK,\RS\bfK \rpar_{L^2(B;\Cbb^{3\times 3})} \Rcb.
\]
with $\bfK(\bfz,\bfy):=\bfD_z\sip\bfq_z\Tsup\sip\bfA^{1/2}\sip\bfG(\bfz,\bfy)\sip\bfA^{1/2}\sip\bfq$. Consequently, the sign heuristic is valid  for any true scatterer $(B,\bfbe)$ and wave number $\kappa$ that satisfy
\begin{equation}
  \|\bfq\sip\bfR_{\kappa}\sip\bfq\| <1, \label{norm:iso:aniso}
\end{equation}

There are other cases when we can conclude the same sign property. For example, in the case where $\bfA_z=\bfA_B$, i.e. $\bfq_z=\bfq$ and $\bfsig_z=\bfsig$ does not seem to provide a clear result as to the sign of $\Tcal(\bfz)$. 

\subsection{Polarization tensor for an ellipsoidal trial inhomogeneity}

We show here that  the integral~\eqref{M:exp} can be evaluated in closed form if $\Bcal$ is an ellipsoid, to obtain
\begin{align}
  \bfM_z
 &= |\Bcal| \lpar \bfI + (\bfA_z\shm\bfA)\sip\bfS\sip\bfA^{-1} \rpar^{-1} \sip(\bfA_z\shm\bfA) \\
 &= 2|\Bcal|\bfA^{1/2}\sip\lpar \bfI\shm\bfQ_z + 2\bfQ_z\sip\bfA^{1/2}\sip\bfS\sip\bfA^{-1/2} \rpar^{-1}\sip\bfQ_z\sip\bfA^{1/2}
\end{align}
where $\bfS$ is the constant Eshelby-like tensor such that $\bfna\bfW_0[\bfg]=\shm\bfS\sip\bfA^{-1}\sip\bfg$ for any $\bfg\shin\Cbb^3$ (with this definition of $\bfS$ mirroring that usually made for elastic inhomogeneities). Moreover, we have $\bfS=(1/3)\bfI$ if $\Bcal$ is the unit ball.

\subsection{Moderate scatterer vs. Born approximation}
\label{Born}

We finish by briefly comparing the domain of validity of the TD heuristics (Theorems~\ref{th-iso} and~\ref{th-iso2}) to that of the Born approximation (BA). For the present physical model defined by~\eqref{eq:propag} and~\eqref{rad:cond}, the BA consists in writing $\bfh\approx(\bfAT\shm\bfA)\cdot\bfna u$, i.e. $u_B\approx u$, for the solution $\bfh$ of~\eqref{h:VIE}, inducing a $O(\|\bfT\shm\bfI\|)$ error on the representation~\eqref{exp:uB:W} of $u\ssup_B$. Now, since $\|\bfna\bfW_{\kappa}\|\geq C>0$ uniformly in $\kappa$ and $B$~\cite[Lemma~3]{B-2016-06}, the weak scatterer condition $\|\bfT\shm\bfI\|=o(1)$ implicit in the BA implies $\|\bfAT\shm\bfA\|=o(\|\bfA\|)$, i.e. $\|\bfbe\|=o(1)$, which in turn implies $\|\bfq\|=o(1)$ and then $\|\bfq\cdot\bfR_{\kappa}\cdot\bfq\|=o(1)$, a condition that is more restrictive than the moderate scatterer limitation $\|\bfq\cdot\bfR_{\kappa}\cdot\bfq\|<1$ of Theorem~\ref{th-iso2} or its isotropic counterpart. Similar conclusions were previously reached in~\cite{B-2013-01} for the case of far-field data and refraction index perturbations.

\section{Conclusion}

We derive an explicit expression of the topological derivative, ${\mathcal T}(\bfz)$, for the scattering by anisotropic media embedded in anisotropic background, with anisotropic trial inhomogeneity of arbitrary shape and near field measurements. Taking advantage of a recently-proposed reformulation of such volume integral equation ~\cite{B-2016-06}, we provide a symmetric factorization for $\Tcal(\bfxZ)$ where the middle operator contains  the material contrast. For the case of isotropic media and background, and isotropic trial inhomogeneity we rigorously prove the sign heuristic for  ${\mathcal T}(\bfz)$. For such configuration, in the particular case of spherical near field measurements far enough form the probing region, we show that $\Tcal(\bfz)=O\left(\frac{1}{(dist(\bfz, B))^{2}}\right)$ if  the location of trial inhomogeneity $\bfz$, is far enough from the unknown scatterer. In the  case of anisotropic media,  we are able to rigorously prove the sign heuristic  for ${\mathcal T}(\bfz)$ only  in some particular case under the general assumption of media with one-sign contrast. Although we are not able to deduce the sign heuristic for the topological derivative for all the combinations of general anisotropic configuration, we remark that our expressions provide a convenient form for  the analysis of the topological derivative, which can possibly be generalized to other type of scattering modalities.

\section*{Acknowledgments}
{ The research of F. Cakoni is partially supported  by the AFOSR Grant  FA9550-17-1-0147 and  NSF Grant DMS-1813492.}

\appendix 

\section{Explicit formulas for special cases}

We present here some explicit examples in the case where $\Gs=R\hatS$, $\hatS$ being the unit sphere, for which explicit analytical results can be derived. The background medium is assumed to be isotropic as described in Section \ref{seciso1}.
\subsection{Far field limit}\label{seca1}
In the  far-field limit when $R\to \infty$ and  for fixed $\kappa$, thanks to the the asymptotic expressions 
\[\begin{aligned}
  \Phik(\bfs\shm\bfz)
 &= \inv{4\pi}\frac{e^{\mathrm{i}\kappa R}}{R} \, e^{-\mathrm{i}\kappa\bfsh\cdot\bfz} + o(|\bfs|^{-1}), \\
  \bfna \Phik(\bfs\shm\bfz)
 &= \frac{\mathrm{i}\kappa}{4\pi}\, \frac{e^{\mathrm{i}\kappa R}}{R} \, e^{-\mathrm{i}\kappa\bfsh\cdot\bfz}\bfsh + o(|\bfs|^{-1}),
\end{aligned} \qquad |\bfs|\to\infty,
\]
we obtain
\[
  \bfG(\bfz,\bfy)
 = \frac{\kappa^2}{4\pi}\Lpar j_0(k|\bfy\shm\bfz|)\widehat{(\bfy\shm\bfz})\otimes\widehat{(\bfy\shm\bfz)} + \frac{j_1(k|\bfy\shm\bfz|)}{k|\bfy\shm\bfz|}[I-3\widehat{(\bfy\shm\bfz})\otimes\widehat{(\bfy\shm\bfz)}] \Rpar 
\]
up to order $o(|\bfs|^{-1})$. This, together with $\bfD$ being real-valued, implies that $\bfK(\bfz,\bfy):=\bfD_z\Tsup\sip\bfG(\bfz,\bfy)$  is real-valued in the far-field limit and that $|\bfK(\bfz,\bfy)|=O(|\bfz\shm\bfy|^{-1})$, yielding the decaying property of the topological derivative stated in Theorem \ref{theo}.

\subsection{Real-valuedness of $\bfG(\bfz,\bfy)$}\label{seca2}

Noting that $\Phik(\bfr)=(\mathrm{i}\kappa/4\pi)h_0^{(1)}(\kappa|\bfr|)$ and recalling a classical expansion of $h^{(1)}_0$ and the Legendre addition theorem, we have
\begin{align}
  \Phik(\bfs\shm\bfz)
 &= \frac{\mathrm{i}\kappa}{4\pi} \sum_{n=0}^{\infty} (2n\shp1) j_n(\kappa|\bfz|) \, h_n^{(1)}(\kappa|\bfs|) P_n(\bfsh\sip\bfzh)
  \qquad (|\bfz|,|\bfy| < |\bfs|=R), \\
 \text{with} \quad
  P_n(\bfsh\sip\bfzh) &= \frac{4\pi}{2n\shp1} \sum_{m=-n}^n Y_n^m(\bfsh)\overline{Y_n^m(\bfzh)}
\end{align}
(where $Y_n^m$ are $L^2(\hatS)$-orthonormal spherical harmonics). The two-point function $L$ can then be evaluated explicitly:
\begin{align}
  L(\bfz,\bfy)
 &= \kappa^2 \sum_{n=0}^{\infty} \sum_{n'=0}^{\infty}  \sum_{m=-n}^n \sum_{m'=-n'}^{n'} \Biggl\{
   \int_{|\bfs|=R} \, h_n^{(1)}(\kappa|\bfs|) \, \overline{h_{n'}^{(1)}(\kappa|\bfs|)} Y_n^m(\bfsh) Y_{n'}^{m'}(\bfsh) \dS(\bfsh)
   \Biggr\} \suite\hspace*{12em}
    j_n(\kappa|\bfz|) \, j_{n'}(\kappa|\bfy|) \, Y_{n'}^{m'}(\bfyh)\overline{Y_n^m(\bfzh)} \\
 &= \kappa^2 R^2 \sum_{n=0}^{\infty} \labs h_n^{(1)}(\kappa R) \rabs^2  \, j_n(\kappa|\bfz|) \, j_n(\kappa|\bfy|) \,
  \sum_{m=-n}^n Y_n^m(\bfyh)\overline{Y_n^m(\bfzh)} \\
 &= \frac{\kappa^2}{4\pi} \sum_{n=0}^{\infty} (2n\shp1) \labs h_n^{(1)}(\kappa R) \rabs^2  \, j_n(\kappa|\bfz|) \, j_n(\kappa|\bfy|) \, P_n(\bfzh\sip\bfyh).
\end{align}
The function $L$ is therefore real-valued, since the $j_n$ are, and~\eqref{G=D2L} implies that $\bfG$ is also real-valued (this observation is corroborated by numerical evaluations using high-accuracy numerical quadrature based on Lebedev points on $\hatS$).

If $\kappa=0$ (in which case $\Phiz$ and $\bfna \Phiz$ are of course real-valued), a similar derivation can be done with $h_n^{(1)}(\kappa R)$ replaced with $R^{-n-1}$ and $j_n$ by a homogeneous $n$-th degree harmonic polynomial (which in particular, unlike $j_n$, is not a decaying function of its argument).
\subsection{Symmetry-restoring operator $E$}\label{seca3}

If $\G$ is a sphere of radius $R$, the ``symmetry-restoring'' operator $E$ can be given an explicit expression. First, for given density $\varphi\in H^{-1/2}(\G)$, the single-layer potential $w:= S\varphi$ solves (see e.g.~\cite{col:83})
\[
  \Delta w + \kappa^2 w = 0 \ \ \text{in }\OO\shcup\lpar\Rbb^3\shsetm\overline{\OO}\rpar, \qquad
  \jump{w}=0 \ \text{and} \ \jump{\partial_n w}=-\varphi \ \ \text{on }\G
\]
When $\G$ is a sphere, the above problem can be solved by separation of variables. Expanding $\varphi$ and $\gamma w=S\varphi$ (where $\gamma$ is the Dirichlet trace operator on $\G$) according to
\[
  \varphi(\bfx) = \sum_{n=0}^{\infty} \sum_{m=-n}^{n} Y^m_n(\bfxh) \varphi^m_n, \quad
  \gamma w(\bfx) = \sum_{n=0}^{\infty} \sum_{m=-n}^{n} Y^m_n(\bfxh) w^m_n \qquad (\bfx\shin\G),
\]
we find
\[
  w^m_n = S^m_n\varphi^m_n, \qquad S^m_n = -\mathrm{i}\kappa(\kappa R)^2 j_n(\kappa R) h^{(1)}_n(\kappa R)
\]
Therefore, since $S^{\star}=\bar{S}$, we have
\[
  E\varphi = \sum_{n=0}^{\infty} \sum_{m=-n}^{n} Y^m_n(\bfxh) E^m_n \varphi^m_n, \qquad
  E^m_n = -\frac{\overline{h^{(1)}_n(\kappa R)}}{h^{(1)}_n(\kappa R)}
\]
In particular, to evaluate $E \Phik(\cdot-\bfxZ)$, we can set
\[
  \varphi^m_n = (2n\shp1) j_n(\kappa|\bfz|) \, h_n^{(1)}(\kappa R) P_n(\bfxh\sip\bfzh),
\]
which implies
\[
  E^m_n \varphi^m_n = -(2n\shp1) j_n(\kappa|\bfz|) \, \overline{h^{(1)}_n(\kappa R)} P_n(\bfxh\sip\bfzh) = -\overline{\varphi^m_n}
\]

 \section*{References}

\bibliographystyle{plain}

\bibliography{biblio,articles,book,bibliomb,theses,fiora}
\end{document}